\documentclass[12pt, twoside, english, headsepline]{article}

\usepackage[margin=30mm]{geometry}

\usepackage{scrlayer-scrpage}  
\clearscrheadfoot                 
\cehead[]{F. Cinque}        
\cohead[]{A Note on the Conditional Probabilities of the Telegraph Process}        
\cefoot[\pagemark]{\pagemark}  
\cofoot[\pagemark]{\pagemark}     

\usepackage{sectsty}
\sectionfont{\fontsize{14}{13}\selectfont}
\subsectionfont{\fontsize{12}{13}\selectfont}

\usepackage{tikz,graphicx}
\usepackage{amsmath,amsthm,amssymb,amsfonts, mathtools, amsbsy, dsfont}
\usepackage[utf8]{inputenx}
\usepackage{xcolor}
\usepackage{microtype}
\usepackage{hyperref}
\hypersetup{linkcolor=blue}

\newcommand*\dif{\mathop{}\!\mathrm{d}}  
\newcommand{\floor}[1]{\left\lfloor #1 \right\rfloor}  
\newcommand*{\doi}[1]{\href{http://dx.doi.org/#1}{DOI: #1}}  

\allowdisplaybreaks 

\newtheorem{theorem}{Theorem}[section] 
\newtheorem{corollary} {Corollary}[section] 
\newtheorem{proposition}{Proposition}[section] 

\theoremstyle{definition} 

\newtheorem{remark}{Remark}[section] 
\numberwithin{equation}{section} 

\providecommand{\keywords}[1]
{
  \small	
  \textbf{\textit{Keywords: }} #1
}
\providecommand{\MSC}[1]
{
  \small	
  \textit{2020 MSC: } #1   
}

\title{A Note on the Conditional Probabilities of the Telegraph Process\footnote{\textit{Submitted to Statistics and Probability Letters; version of February 3, 2022}}}
\author{Fabrizio Cinque\\
        \small Department of Statistical Sciences, Sapienza University of Rome, Italy, \\
        \small fabrizio.cinque@uniroma1.it
}

\date{} 

\begin{document}

\maketitle

\begin{abstract}
We consider the telegraph process with two velocities, $a_1>a_2\in\mathbb{R}$, and two rates of reversal, $\lambda_1,\lambda_2>0$. We study some of its features with respect to the conditional probability measure where both the initial speed and the number of changes of direction are known. We exhibit a new proof by induction of the (conditional) probability law and a detailed study of the distribution of the motion at time $t>0$ conditioned on its position at a previous time $0<s<t$. In the case of a symmetric process, we present some results on the joint distribution of the position of the motion at time $t>0$, its maximum and its minimum up to that moment.
\end{abstract}


\MSC{Primary 60K99; 60G50}

\section{Introduction}

Let $V_0$ be a uniformly distributed random variable taking values $a_1,a_2\in \mathbb{R}$, with $a_1>a_2$. Let $N=\{N(t)\}_{t\ge0}$ be a non decreasing counting process. We define the random velocity process $\{V(t)\}_{t\ge0}$  as
\begin{equation}
V(t) := \frac{a_1+a_2}{2} + \Bigl(V_0-\frac{a_1+a_2}{2} \Bigr)(-1)^{N(t)},
\end{equation}
thus $V(t)\in \{a_1,a_2\}\ a.s.$ and $V(t) = V(0)=V_0$ if and only if $N(t)$ is even.
The telegraph process $\mathcal{T} =\{\mathcal{T}(t)\}_{t\ge0}$ is defined as
\begin{equation}\label{definizioneTelegrafo}
 \mathcal{T}(t) := \int_0^tV(s)\dif s = \sum_{i=0}^{N(t)-1} \Bigl(T_{i+1}-T_i\Bigr) V(T_i)+ \Bigl(t-T_{N(t)}\Bigr)V(T_{N(t)}),
\end{equation}
where $T_n,\ n\in \mathbb{N}$, is the $n$-th arrival time of $N$, $T_0 = 0\ a.s.$ and $V(T_n)$ is the velocity after the $n$-th change of direction. $\mathcal{T}(t)$ represents the position, at time $t$, of a particle which starts moving with velocity $V(0) =V_0$ at time $t=0$ and changes its speed when an event of the point process $N$ occurs.
As usual, we assume that $N$ is a Poisson-type process with independent and exponentially distributed waiting times with two alternating parameters, $\lambda_1>0$ if $V(t)=a_1$ and $\lambda_2>0$ if $V(t)=a_2$. Therefore, the telegraph motion proceeds with velocity $a_1$ ($a_2$) for an exponential random time of average $1/\lambda_1$ ($1/\lambda_2$). $\mathcal{T}(t)$ is in $\{a_1t,a_2t\}$ if no changes of direction occur in $[0,t]$. Thus, $P\{\mathcal{T}(t)=a_it\} = P\{V(0)=a_i,N(t) = 0\}=e^{-\lambda_it}/2$.
\\Throughout the paper, we call $\mathcal{T}$ \textit{asymmetric telegraph process} if $N$ is a homogeneous Poisson process ($\lambda_1=\lambda_2$). If there is also $a_1=-a_2>0$, then, we call $\mathcal{T}$ \textit{symmetric telegraph process}.
\\

The (symmetric) telegraph process has been introduced by Goldstein (1951). Its transition density satisfies the following partial differential equation (which, in the symmetric case, $\lambda_1=\lambda_2,\ a_1=-a_2>0$, reduces to the telegraph equation)
\begin{equation}\label{equazioneTelegrafoGenerale}
\frac{\partial^2 p}{\partial t^2}+(a_1+a_2)\frac{\partial^2 p}{\partial t\partial x} +(\lambda_1+\lambda_2) \frac{\partial p}{\partial t} =-a_1a_2\frac{\partial^2 p}{\partial x^2}+\frac{(a_1+a_2)(\lambda_1+\lambda_2)+(a_1-a_2)(\lambda_1-\lambda_2)}{2}\frac{\partial p}{\partial x}.
\end{equation}
For $a_1>0>a_2$, under the generalized Kac's conditions $\lambda_1,\lambda_2,a_1,-a_2\longrightarrow\infty$ such that $\lambda_1/\lambda_2\longrightarrow1,\ -a_1a_2/(\lambda_1+\lambda_2)\longrightarrow\sigma^2/2>0$ (or alternatively $-a_1a_2/\sqrt{\lambda_1 \lambda_2}\longrightarrow \sigma^2$) and $(a_1\lambda_1+a_2\lambda_2)/(\lambda_1+\lambda_2) \longrightarrow-\mu\in\mathbb{R}$, equation (\ref{equazioneTelegrafoGenerale}) converges to the heat equation with drift $\mu$ and diffusivity $\sigma^2/2$ and the telegraph process converges in distribution to standard Brownian motion with drift $\mu$ and diffusion $\sigma^2$.
Several authors used this connection with the partial differential equations to obtain some important results, see for instance Kac (1974), Orsingher (1990) and Kolesnik (1998). The distribution of the telegraph process with two different rates and velocities has been first proved in Beghin \textit{et al.} (2001) by means of two different approaches, one of which based on the governing partial differential equation (\ref{equazioneTelegrafoGenerale}). Among others, some remarkable results were proved by Di Crescenzo (2001), Stadje and Zacks (2004), Zacks (2004) and more recently by Di Crescenzo \textit{et al.} (2013) and Lopez and Ratanov (2014). In the latter, the authors show also the explicit probabilities of the telegraph process conditioned on both the starting velocity and the number of changes of direction. In the second section of this paper, we show an alternative proof of these distributions.

The telegraph motion has a finite velocity and it is suitable to describe real motions that emerge in different fields, like geology, see Travaglino \textit{et al.} (2018), finance, see Di Crescenzo and Pellerey (2002), Ratanov (2007), Kolesnik and Ratanov (2013), and physics, see Hartmann \textit{et al.} (2020).

The telegraph process is non-Markov, while the couple $\{\mathcal{T}(t),V(t)\}_{t\ge0}$ is it. This fact can be found in Davis (1984) and it has been used to infer the rate parameter of the motion, see De Gregorio and Iacus (2008).
In Section 3, we display the exact distribution of $\mathcal{T}(t)$ conditioned on the position at previous time $\mathcal{T}(s)=x$ with $0\le s<t$ and $x\in[a_2s,a_1s]$. This conditional law may be very useful for the parametric estimation of the motion, for instance, by applying it in a pairwise composite likelihood.

At last, in the fourth section we study the conditional joint distribution of the symmetric telegraph process, its minimum and its maximum. We provide a general recurrent formula and explicit results under some particular condition which also lead to a reflection-type principle resembling the negative reflection principle, introduced in Cinque (2020), concerning finite velocity random motions. We recall that several papers have been devoted to the analysis of the first passage times, see Foong (1992), Foong and Kanno (1994), Orsingher (1995) and more recently of Lopez and Ratanov (2014), Mori \textit{et al.} (2020), De Bruyne \textit{et al.} (2021) and Ratanov (2021). The investigation of the motion in presence of multiple boundaries has been recently carried out by Di Crescenzo \textit{et al.} (2018) and Di Crescenzo \textit{et al.} (2020). On the other hand, the maximum has been further studied in Cinque and Orsingher (2020, 2021) and Cinque (2020).


\section{Conditional probability law of the telegraph process}

In this section we show a new induction approach to study the probability of the telegraph process $\mathcal{T}$, given the initial velocity and the number of changes of direction. This method permits us to obtain explicit conditional probabilities concerning the sum of the odd waiting times, $T_{2k+1}-T_{2k}, k\in \mathbb{N}_0,$ of the Poisson-type process $N =\{N(t)\}_{t\ge0}$, governing the switches of the random motion. It is interesting to show that these results also find an application in the inference of the rates of $N$.
\\

We begin by studying the probability mass function of the Markovian process $N$ with (ordered) alternating parameters $\lambda_1,\lambda_2>0$. By means of formulas (3.7), (3.8) and Remark 3.3 of Cinque (2022) it is easy to derive that, for $k \in \mathbb{N}_0$,
\begin{equation}\label{tsrPppari}
P\{N(t) = 2k\} = (\lambda_1 t)^k (\lambda_2 t)^k e^{-\lambda_1 t}\,E^{k}_{1,2k+1}\Bigl(t(\lambda_1-\lambda_2) \Bigr),
\end{equation}
\begin{equation}\label{tsrPpdispari}
P\{N(t) = 2k+1\} = (\lambda_1 t)^{k+1} (\lambda_2 t)^k e^{-\lambda_1 t}\,E^{k+1}_{1,2k+2}\Bigl(t(\lambda_1-\lambda_2) \Bigr),
\end{equation}
where $E_{\nu, \delta}^{\gamma} (x) = \sum_{j=0}^\infty \frac{\Gamma(\gamma+j)}{\Gamma(\gamma)\,j!}\frac{x^j}{\Gamma(\nu j+ \delta)}$, with $\nu, \gamma,\delta \in \mathbb{C}, Re(\nu), Re(\gamma), Re(\delta) >0$ and $x\in \mathbb{R}$, is the generalized Mittag-Leffler function. It can be proved also that, for $0<t_l<t_m<t$ with $n\ge m> l \in \mathbb{N}$,
\begin{align}
P \{T_{l}\in \dif t_l,\ T_{m}\in &\dif t_m\ |\ N(t) = n \} = \frac{t_l^{l-1}(t_m-t_l)^{m-l-1}(t-t_m)^{n-m}}{t^{n}}\:E^{\floor{\frac{l}{2}}}_{1,l}\Bigl(t_l(\lambda_1-\lambda_2) \Bigr) \nonumber\\
&\times \,\frac{E^{\floor{\frac{m}{2}}-\floor{\frac{l}{2}}}_{1,m-l}\Bigl((t_m-t_l)(\lambda_1-\lambda_2) \Bigr)\,E^{\floor{\frac{n+1}{2}}-\floor{\frac{m}{2}}}_{1,n+1-m}\Bigl((t-t_m)(\lambda_1-\lambda_2) \Bigr)}{E^{\floor{\frac{n+1}{2}}}_{1,n+1}\Bigl(t(\lambda_1-\lambda_2) \Bigr)} \dif t_l\dif t_m.\label{dueTempiTsrPp}
\end{align}

\begin{remark}\label{remarkTelegrafoSommeDifferenzeAlternate}
Let $V(0) = v_0,\ v_0\in \{a_1,a_2\}$, and $v_1$ be the other possible speed. By conditioning on the initial velocity $V(0)=v_0$ and the number of changes of direction $N(t) = n\in \mathbb{N}$, in view of (\ref{definizioneTelegrafo}), we have
\begin{align}\label{relazioneTelegrafoSommeDifferenzeAlternateDispari}
n = 2k+1 \implies \mathcal{T}(t) &= v_0T_1+v_1\bigl(T_2-T_1\bigr) +\dots + v_0\bigl(T_{2k+1}-T_{2k}\bigr)+v_1\bigl(t-T_{2k+1}\bigr) \nonumber\\
&=(v_0-v_1)\bigl(T_1-T_2+T_3-\dots -T_{2k}+T_{2k+1} \bigr) +v_1t 
\end{align}
and similarly in the even case $n = 2k \implies \mathcal{T}(t) =(v_0-v_1)\bigl(T_1-T_2+T_3-\dots +T_{2k-1}-T_{2k} \bigr) +v_0t$.
\\Thus, conditionally on the number of switches in $(0,t)$ and the starting speed, the telegraph process at time $t$ is an affine transformation of the ``alternating sums and differences'' of the arrival times of $N(t)$. 
\\Let $N(t)\ge n\in \mathbb{N}$, we define $S_{n}(t) := \sum_{i=1}^{n}T_i(-1)^{i-1}$. Note that, if $N(t) = 2k+1$, then $S_{2k+1}(t) = T_1-T_2+T_3-\dots -T_{2k}+T_{2k+1} = T_1+(T_3-T_2)+\dots+(T_{2k+1}-T_{2k})$ represents the time that the telegraph process spends with speed $V(0)$ and $S_{2k+1}(t)\in (0,t)\ a.s.$. On the other hand, if $N(t) =2k$, $S_{2k}(t)$ represents the time that the particle spends with speed $V(0)$ minus the total time $t$, then $S_{2k}(t)\in (-t,0)\ a.s.$
\hfill$\diamond$
\end{remark}

Henceforth, we use the following notation. Let $(\Omega, \mathcal{F}, P)$ be a probability space and $X,Y:\Omega \longmapsto \mathbb{R}$ be absolutely continuous random variables with joint probability density $f_{X,Y}$, for $a,b,c,d,x,y\in \mathbb{R}, a,c\not = 0$, we write $P\{X\in (\dif x-b)/a, Y\in (\dif y-d)/c\} := f_{X,Y}\bigl((x-b)/a, (y-d)/c\bigr) \dif x\dif y / |a\,c| = P\{aX+b\in \dif x,cY+d\in \dif y\}$.

\begin{theorem}\label{teoremaSommeDifferenzeAlternate}
Let $\{N(t)\}_{t\ge0}$ be a Poisson process with ordered rates $(\lambda_1,\lambda_2)$ and $S_n(t)$ defined as above. For $k \in \mathbb{N}, \ s\in (-t,0)$,
\begin{equation}\label{probSommeDifferenzeAlternatePari}
P\{S_{2k}(t)\in \dif s\ |\ N(t) = 2k\} = \frac{e^{-(\lambda_1-\lambda_2)s}\:(-s)^{k-1}(t+s)^k}{E^{k}_{1,2k+1}\Bigl(t(\lambda_1-\lambda_2)\Bigr)\, k!(k-1)!\, t^{2k}}\dif s
\end{equation}
and, for $k \in \mathbb{N}_0, \ s\in (0,t)$,
\begin{equation}\label{probSommeDifferenzeAlternateDispari}
P\{S_{2k+1}(t)\in \dif s\ |\ N(t) = 2k+1\} = \frac{e^{(\lambda_1-\lambda_2)(t-s)}\:s^{k}(t-s)^k}{E^{k+1}_{1,2k+2}\Bigl(t(\lambda_1-\lambda_2)\Bigr)\, k!^2\, t^{2k+1}}\dif s,
\end{equation}
where $E$ is the generalized Mittag-Leffler function.
\end{theorem}

\begin{proof}
We prove (\ref{probSommeDifferenzeAlternatePari}) by induction. We begin by showing that the following equation is true for $k \in \mathbb{N},\, t>0$.
\begin{align}\label{ipotesiInduttivaProbSommeDifferenzeAlternatePari}
 &P\{S_{2k}(t)\in \dif s\ |\ N(t) = 2k\} \\
&= \frac{e^{-(\lambda_1-\lambda_2)s}}{E^{k}_{1,2k+1}\Bigl(t(\lambda_1-\lambda_2)\Bigr) t^{2k}}\dif s\int_0^{t+s}\dif s_{2k-1}\int_s^0\dif s_{2k-2}\int_0^{s_{2k-1}}\dif s_{2k-3}\int_{s_{2k-2}}^0\dif s_{2k-4} \ \cdot \ \cdot \ \cdot \int_{s_4}^0\dif s_2 \int_0^{s_3} \dif s_1.\nonumber
\end{align}
Formula (\ref{ipotesiInduttivaProbSommeDifferenzeAlternatePari}) holds for $k=1$, in fact, for $s\in (-t, 0)$,
$$ P\{T_1-T_2\in \dif s\ |\ N(t) = 2\} = \int_0^{t+s}P\{T_1 \in \dif t_1,\ T_2 \in t_1-\dif s\ |\ N(t) = 2\}$$
and the base case is proved by suitably using (\ref{dueTempiTsrPp}). Let us suppose that (\ref{ipotesiInduttivaProbSommeDifferenzeAlternatePari}) holds for $k-1$ and note that, if $2h<n\in\mathbb{N}$ and real $0<-s<r<t,\; P\{S_{2h}(t)\in\dif s|N(t)=n, T_{2h}=r\}=P\{S_{2h}(r)\in\dif s|N(r)=2h\}$. 
Now, for $s\in(-t,0)$,
\begin{align}
P\{&S_{2k}(t)\in \dif s\ |\ N(t) = 2k\} = \int_0^{t} P\{S_{2k}(t)\in \dif s, \ S_{2k-1}(t)\in \dif s_{2k-1}\ |\ N(t) = 2k\}  \nonumber\\
& = \int_0^{t+s}P\{-T_{2k}\in \dif s- s_{2k-1}, \ S_{2k-1}(t)\in \dif s_{2k-1}\ |\ N(t) = 2k\}  \label{primoPassaggio}\\
& = \int_0^{t+s}\int_s^0 P\{T_{2k}\in s_{2k-1}- \dif s, T_{2k-1}\in \dif s_{2k-1}-s_{2k-2},\  S_{2k-2}(t)\in \dif s_{2k-2}\ |\ N(t) = 2k\} \label{secondoPassaggio}\\
& = \int_0^{t+s}\int_s^0 P\{T_{2k-1}\in  \dif s_{2k-1}-s_{2k-2},\ T_{2k}\in s_{2k-1}- \dif s\ |\  N(t) = 2k\} \nonumber\\
& \hspace{2cm}\times P\{S_{2k-2}(s_{2k-1}-s_{2k-2})\in \dif s_{2k-2}\ |\ N(s_{2k-1}-s_{2k-2}) = 2k-2\}, \label{ultimoPassaggio}
\end{align}
where the integration sets of (\ref{primoPassaggio}) and (\ref{secondoPassaggio}) follow by respectively considering that $s_{2k-1}\in(0,t)$ and $s_{2k-1}-s\in (s_{2k-1}, t)$  and that $s_{2k-2}\in (-t, 0)$ and $s_{2k-1}-s_{2k-2}\in(-s_{2k-2},\, s_{2k-1}-s)$. By suitably plugging the induction hypothesis (\ref{ipotesiInduttivaProbSommeDifferenzeAlternatePari}) (with $k-1$) and (\ref{dueTempiTsrPp}) into (\ref{ultimoPassaggio}), (\ref{ipotesiInduttivaProbSommeDifferenzeAlternatePari}) shows up.

Finally, thanks to Fubini's theorem, we can write (\ref{ipotesiInduttivaProbSommeDifferenzeAlternatePari}) by separating the integrals with odd and even indexed variables
\begin{align*}
 &P\{S_{2k}(t)\in \dif s\ |\ N(t) = 2k\}  \\
&= \frac{e^{-(\lambda_1-\lambda_2)s}}{E^{k}_{1,2k+1}\Bigl(t(\lambda_1-\lambda_2)\Bigr) t^{2k}}\dif s \int_0^{t+s}\dif s_{2k-1} \int_0^{s_{2k-1}}\dif s_{2k-3}\cdots \int_0^{s_3}\dif s_1 \, \int_s^0 \dif s_{2k-2} \int_{s_{2k-2}}^0\dif s_{2k-4}\cdots \int_{s_4}^0 \dif s_2.\nonumber
\end{align*}
We conclude the proof of probability (\ref{probSommeDifferenzeAlternatePari}) by using the Cauchy integral formula, i.e. $\int^x_0 \dif x_1 \int^{x_1}_0 \dif x_2 \int^{x_2}_0  \cdot \cdot \cdot \int^{x_{n-1}}_0 f(x_n) \dif x_n =  \int_0^x f(x_n) (x-x_n)^{n-1} \dif x_n / \Gamma(n)$ with $f$ any integrable (real) function, $x\in \mathbb{R}$ and $n\in \mathbb{N}$.

The proof of (\ref{probSommeDifferenzeAlternateDispari}) works in the same way, but by means of the following induction hypothesis, for $k\in \mathbb{N}_0$ and $s\in (0,t)$,
\begin{align*}\label{ipotesiInduttivaProbSommeDifferenzeAlternateDispari}
& P\{S_{2k+1}(t)\in \dif s\ |\ N(t) = 2k+1\}\\
& = \frac{e^{(\lambda_1-\lambda_2)(t-s)}}{E^{k+1}_{1,2k+2}\Bigl(t(\lambda_1-\lambda_2)\Bigr) t^{2k+1}}\dif s\int^0_{s-t}\dif s_{2k}\int_0^s\dif s_{2k-1}\int^0_{s_{2k}}\dif s_{2k-2}\int^{s_{2k-1}}_0\dif s_{2k-3} \cdots\int_{s_4}^0\dif s_2 \int_0^{s_3} \dif s_1.
\end{align*}
\end{proof}

\begin{remark}\label{remarkMomenti}
Let $n \in \mathbb{N}$ and $\mathbb{E}_n[\cdot]$ be the expected value with respect to the probability measure $P\{\cdot \,|\,N(t) = n\}$. Put $S_n=S_n(t)$. The $m$-th moment is, with $m>-k$ if $n=2k$ and $m>-k-1$ if $n=2k+1$,
$$\mathbb{E}_{2k} \bigl[S_{2k}^m\bigr] =  (-t)^m\frac{\Gamma(k+m)}{(k-1)!}\frac{E^{k+m}_{1,2k+1+m}\Bigl(t(\lambda_1-\lambda_2)\Bigr)}{E^{k}_{1,2k+1}\Bigl(t(\lambda_1-\lambda_2)\Bigr)},$$
$$   \mathbb{E}_{2k+1} \bigl[S_{2k+1}^m\bigr] =  t^m\,\frac{\Gamma(k+m+1)}{k!}\frac{E^{k+1}_{1,2k+2+m}\Bigl(t(\lambda_1-\lambda_2)\Bigr)}{E^{k+1}_{1,2k+2}\Bigl(t(\lambda_1-\lambda_2)\Bigr)}.$$
Clearly, for $m=0$ these expected values are both equal to $1$.
\hfill$\diamond$
\end{remark}

\begin{remark}
Theorem \ref{teoremaSommeDifferenzeAlternate} is useful to infer the ordered rates $(\lambda_1,\lambda_2)$ of the Poisson process $N$, whose probability mass appears in (\ref{tsrPppari}) and (\ref{tsrPpdispari}). In fact, by observing the events recorded up to time $t>0$, the maximum likelihood estimators are (put $S(t) = S_{N(t)}(t)$)
\\$ \hat \lambda_1= \mathds{1}\bigl(N(t)\ \text{even}\bigr)\frac{N(t)/2}{t+S(t)}\,+\,\mathds{1}\bigl(N(t)\ \text{odd}\bigr)\frac{N(t)+1}{2\, S(t)},\ \ \hat \lambda_2= \mathds{1}\bigl(N(t)\ \text{even}\bigr)\frac{N(t)/2}{- S(t) }\,+\,\mathds{1}\bigl(N(t)\ \text{odd}\bigr)\frac{N(t)-1}{2(t-S(t))}.$
\end{remark}

\begin{theorem}\label{teoremaTelegrafoGeneralizzatoDatoTsrPpV}
Let $\{\mathcal{T}(t)\}_{t\ge0}$ be a telegraph process. Let $x\in (a_2t, a_1t)$. For $k\in \mathbb{N}$,
\begin{equation}\label{probTelegrafoGeneralizzatoPari+}
P\{\mathcal{T}(t)\in \dif x\ |\ N(t) = 2k,\ V(0) = a_1\} = \frac{e^{(\lambda_1-\lambda_2)\frac{a_1t-x}{a_1-a_2}}}{E^{k}_{1,2k+1}\Bigl(t(\lambda_1-\lambda_2)\Bigr)}\frac{(a_1t-x)^{k-1}(x-a_2t)^{k}}{k!(k-1)!\bigl[(a_1-a_2)t\bigr]^{2k}}\dif x,
\end{equation}
\begin{equation}\label{probTelegrafoGeneralizzatoPari-}
P\{\mathcal{T}(t)\in \dif x\ |\ N(t) = 2k,\ V(0) = a_2\} =\frac{e^{(\lambda_1-\lambda_2)\frac{a_1t- x}{a_1-a_2}}}{E^{k+1}_{1,2k+1}\Bigl(t(\lambda_1-\lambda_2)\Bigr)}\frac{\bigl(a_1t-x\bigr)^k\bigl(x-a_2t\bigr)^{k-1}}{k!(k-1)!\, \bigl[(a_1-a_2)t\bigr]^{2k}} \dif x,
\end{equation}
and for $k\in\mathbb{N}_0,\ v_0\in \{a_1,a_2\}$,
\begin{equation}\label{probTelegrafoGeneralizzatoDispari}
P\{\mathcal{T}(t)\in \dif x\ |\ N(t) = 2k+1,\ V(0) = v_0\} =
\frac{e^{(\lambda_1-\lambda_2)\frac{a_1t-x}{a_1-a_2}}}{E^{k+1}_{1,2k+2}\Bigl(t(\lambda_1-\lambda_2)\Bigr)}\frac{(a_1t-x)^{k}(x-a_2t)^{k}}{k!^2\bigl[(a_1-a_2)t\bigr]^{2k+1}}\dif x.
\end{equation}
\end{theorem}
The distributions in Theorem \ref{teoremaTelegrafoGeneralizzatoDatoTsrPpV} can be found in Lopez and Ratanov (2014) (where the confluent hypergeometric Kummer function is involved). 

\begin{proof}
The proof follows by considering Remark \ref{remarkTelegrafoSommeDifferenzeAlternate}, Theorem \ref{teoremaSommeDifferenzeAlternate} and that $N$ has parameter $(\lambda_1,\lambda_2)$ if $V(0)=a_1$ and $(\lambda_2,\lambda_1)$ if $V(0)=a_2$. Also note that $E_{1, \gamma_1+\gamma_2}^{\gamma_1}(y-x) = e^{y-x}E_{1, \gamma_1+\gamma_2}^{\gamma_2}(x-y)$, with $\gamma_1,\gamma_2\in \mathbb{C},\ Re(\gamma_1), Re(\gamma_2)>0$ and $x,y\in \mathbb{R}$ (see Remark 3.3 of Cinque (2022)).
\end{proof}
Note that if $N(t) = 2k+1$, the motion performs $k+1$ displacements with each velocity and, thanks to the Markovianity of $N$, the initial velocity is not relevant. For instance, if we consider $N$ with Mittag-Leffler distributed waiting times, the process loses this property, see Theorem 3.4 of Cinque (2022).

From Theorem \ref{teoremaTelegrafoGeneralizzatoDatoTsrPpV}, with (\ref{tsrPppari}) and (\ref{tsrPpdispari}) at hand, we obtain the well-known absolutely continuous component of the telegraph process. Let $v_0\in\{a_1,a_2\}$ and $v_1$ the other possible speed. For $x\in(a_2t,a_1t)$
\begin{align}\label{telegrafoV}
&P\{\mathcal{T}(t)\in \dif x\,|\,V(0)=v_0\}  = \frac{e^{{-\frac{\lambda_1(x-a_2t)+\lambda_2(a_1t-x)}{a_1-a_2}}}}{a_1-a_2}\dif x\\
& \times \Biggl[\lambda_h\, I_0\Bigl(\frac{2\sqrt{\lambda_1\lambda_2}}{a_1-a_2}\sqrt{(a_1t-x)(x-a_2t)}\Bigr)+\sqrt{\lambda_1\lambda_2}\sqrt{\frac{|v_1t-x|}{|v_0t-x|}}\, I_1\Bigl(\frac{2\sqrt{\lambda_1\lambda_2}}{a_1-a_2}\sqrt{(a_1t-x)(x-a_2t)}\Bigr)\Biggr],\nonumber
\end{align}
where $h = \Big\{ \begin{array}{l l}1 & \text{if}\ v_0=a_1 \\2 & \text{if}\ v_0=a_2 \end{array}$ and $I_\nu(x) = \sum_{k=0}^\infty\frac{(x/2)^{2k+\nu}}{k!\Gamma(k+1+\nu)}$ is the modified Bessel function of order $\nu\in \mathbb{R}$.

\section{Probability law of the telegraph process conditioned on the position at previous time}

The counting process $N$ has independent waiting times, thus, every time a change of direction occurs, the motion $\mathcal{T}$ starts again with the same characteristics, independently on the previous displacements (except for the order of the rates, given by the current velocity). As above, we denote with $T_k<t$ the $k$-th arrival time of the process $N(t)=n\ge k\in \mathbb{N}$ and with $V(T_k) = v_k$ the speed at time $T_k$ after the change of direction. It is clear that for $x \in [a_2T_k,a_1T_k]$ and $y\in\bigr(x+a_2(t-T_k),x+a_1(t-T_k)\bigl)$,
\begin{align}\label{telegrafoCondizionatoTempiArrivo}
P\{ \mathcal{T}(t) \in \dif y\,& |\, \mathcal{T}(T_k) = x,\, N(T_k) = k,\, N(t) = n,\, V(0) = v_0  \}\\
&= P \{ \mathcal{T}(t-T_k) \in \dif y-x\, |\, N(t-T_k) = n-k,\, V(0) =  v_k\}.\nonumber
\end{align}

Furthermore, when $N$ is Markovian, the couple $\big\{\bigl(\mathcal{T}(t),V(t)\bigr)\big\}_{t\ge0}$ is a Markov process. Now, by considering that the event $\{N(t)=k, V(0) = v_0\}$ implies that $V(t) = v_k$ is known, we immediately obtain the following result (see Appendix A for an alternative proof of Theorem \ref{corollarioTelegrafoCondizionatoNsV}).

\begin{theorem}\label{corollarioTelegrafoCondizionatoNsV}
Let $\{\mathcal{T}(t)\}_{t\ge0}$ be a telegraph process. Let $0<s<t, \ x\in  [a_2s,a_1s]$ and $v_0\in\{a_1,a_2\}$. For $k\in \mathbb{N}_0,\ (y-x)\in\bigr(a_2(t-s),a_1(t-s)\bigl)$,
\begin{equation}\label{telegrafoCondizionatoNsV}
P\{ \mathcal{T}(t) \in \dif y\ |\ \mathcal{T}(s) = x, N(s) = k,V(0) = v_0  \}= P \{ \mathcal{T}(t-s) \in \dif y-x\ |\ V(0) =  v_k\}
\end{equation}
and
\begin{equation}\label{telegrafoCondizionatoNsVSingolarita}
P\{ \mathcal{T}(t) =x+v_k(t-s) \ |\ \mathcal{T}(s) = x,N(s) = k, V(0) = v_0  \} = P\{N(t-s) = 0\ |\ V(0) = v_k\},
\end{equation}
with $v_k = V(T_k)$ being the speed after the $k$-th change of direction.
\end{theorem}

We consider the asymmetric telegraph process ($\lambda_1 = \lambda_2 = \lambda$). The interested reader can obtain the results for the telegraph process with two different rates by following the same steps we show below.

\begin{remark}\label{remarkVDatoTNt}
We recall that if $\{\mathcal{T}(t)\}_{t\ge0}$ is an asymmetric telegraph process, then, for $n\in \mathbb{N},\ x\in  (a_2t,a_1t),\ v_0\in\{a_1,a_2\}$ and $v_1$ being the other possible velocity,
\begin{equation}\label{probVDatoTNt}
P \{ V(0) = v_0\ |\ \mathcal{T}(t) = x, N(t) = n\} =
\begin{cases}\begin{array}{l l} 
 \frac{|v_1t-x|}{(a_1-a_2)t} & \text{if} \ n\ \text{even,} \\
\frac{1}{2} & \text{if} \  n\ \text{odd.} 
\end{array}
\end{cases}
\end{equation}
Note that the odd case in (\ref{probVDatoTNt}) holds true also for the telegraph process with two different rates.

We point out that for $k\in \mathbb{N},\ P \{ V(0) = a_1\ |\ \mathcal{T}(t) = x, N(t) = 2k\} \ge\frac{1}{2}\iff x\ge(a_1+a_2)t/2$. Clearly, $V(0) = V(t)$ because an even number of changes of direction occurred in $[0,t]$ and thus, at time $t$, it is more likely that the motion has velocity $a_1$ if and only if $|x-a_1t|<|x-a_2t|$. Note that in the case of the symmetric telegraph process ($a_1 = -a_2>0$) $P \{ V(0) = v_0\ |\ \mathcal{T}(t) = x, N(t) = 2k\} >\frac{1}{2}\iff sign(v_0) = sign(x)$.
\hfill $\diamond$
\end{remark}

By means of simple probability elaborations, the previous results permit us to obtain Theorem \ref{corollarioTelegrafoCondizionatoDatoNs} and Corollary \ref{corollariotcTs}. These provide the explicit formulas for the conditional distributions of the telegraph process as the sum of the ``Markov'' term, that is the probability that would appear if Markovianity held, and an additional term.

\begin{theorem}\label{corollarioTelegrafoCondizionatoDatoNs}
Let $\{\mathcal{T}(t)\}_{t\ge0}$ be an asymmetric telegraph process. Let $0<s<t,\ x \in (a_2s,a_1s),\ y\in \bigl(x+a_2(t-s), x+a_1(t-s)\bigr), \ v_0\in\{a_1,a_2\}$ and $v_1$ be the other possible speed. For $k\in \mathbb{N}_0$,
\begin{equation}\label{telegrafoCondizionatoDatoNsDispari}
P\{ \mathcal{T}(t)\in \dif y\ |\ \mathcal{T}(s) = x, N(s) = 2k+1\} = P\{\mathcal{T}(t-s) \in \dif y-x\}
\end{equation}
and
\begin{equation}\label{telegrafoCondizionatoDatoNsDispari0}
P\{\mathcal{T}(t) =x +v_0(t-s)\ |\ \mathcal{T}(s) = x, N(s) = 2k+1  \} =  P\{\mathcal{T}(t-s) = v_0(t-s)\}.
\end{equation}
For $k\in \mathbb{N}$,
\begin{equation}\label{telegrafoCondizionatoDatoNsPari}
P\{ \mathcal{T}(t)\in \dif y\ |\ \mathcal{T}(s) = x, N(s) = 2k\} =P\{\mathcal{T}(t-s) \in \dif y-x\} +  g(s,t-s;x,y-x)\dif y,
\end{equation}
where
\begin{align}
 g(s,&t-s;x,y-x) =\frac{I_1\Bigl( \frac{2\lambda}{a_1-a_2}\sqrt{\bigl[a_1(t-s)-(y-x)\bigr]\bigl[(y-x)-a_2(t-s)\bigr]}\Bigr)}{\sqrt{\bigl[a_1(t-s)-(y-x)\bigr]\bigl[(y-x)-a_2(t-s)\bigr]}}\label{funzioneg}\\
 &\ \ \ \ \times\frac{\lambda e^{-\lambda (t-s)}}{2(a_1-a_2)^2s}\Bigl[ 4x(y-x)+(a_1+a_2)\bigl[ (a_1+a_2)s(t-s)-2s(y-x)-2(t-s)x\bigr]\Bigr]\nonumber\\
 & = e^{-\lambda (t-s)}\,\frac{(a_1+a_2)s-2x}{2(a_1-a_2)s}\,\frac{\partial}{\partial y}I_0\Bigl( \frac{2\lambda}{a_1-a_2}\sqrt{\bigl[a_1(t-s)-(y-x)\bigr]\bigl[(y-x)-a_2(t-s)\bigr]}\Bigr)\label{gConDerivata}
\end{align}
and
\begin{equation}\label{telegrafoCondizionatoDatoNsPari0}
P\{\mathcal{T}(t) =x +v_0(t-s)\ |\ \mathcal{T}(s) = x, N(s) = 2k  \} = \frac{|v_1s-x|}{(a_1-a_2)s} P\{\mathcal{T}(t-s) = v_0(t-s)\} .
\end{equation}
\end{theorem}

Note that all the results are independent of $k$. Moreover, (\ref{telegrafoCondizionatoDatoNsDispari}) and (\ref{telegrafoCondizionatoDatoNsDispari0}) hold true for the telegraph process with two different rates.

\begin{proof}
The proof follows from Theorem \ref{corollarioTelegrafoCondizionatoNsV} and Remark \ref{remarkVDatoTNt}. To prove (\ref{telegrafoCondizionatoDatoNsPari}), put $P_i\{\, \cdot\,\} =  P\{\, \cdot\, |\, V(0) = a_i \},\ i=1,2$, then
\begin{align} 
P&\{ \mathcal{T}(t)\in \dif y\ |\ \mathcal{T}(s) = x, N(s) = 2k\} =\frac{(x-a_2s)P_1\{\mathcal{T}(t-s) \in \dif y-x\} +(a_1s-x)P_2\{\mathcal{T}(t-s) \in \dif y-x\} }{(a_1-a_2)s}\nonumber\\
&=2P\{\mathcal{T}(t-s) \in \dif y-x\} - \frac{(a_1s-x)P_1\{\mathcal{T}(t-s) \in \dif y-x\} +(x-a_2s)P_2\{\mathcal{T}(t-s) \in \dif y-x\} }{(a_1-a_2)s}\nonumber
\end{align}
and by suitably applying formula (\ref{telegrafoV}), some calculation yield the claimed result.
\end{proof}

By taking into account the asymptotic behavior of the Bessel function, for $\nu=0,1$, $I_{\nu}(x)\sim e^x/\sqrt{2\pi x}$, with $x\longrightarrow \infty$, we can prove that (\ref{funzioneg}) converges to $0$ under the generalized Kac's conditions (see above, after equation (\ref{equazioneTelegrafoGenerale})). Furthermore, thanks to expression (\ref{gConDerivata}) we readily obtain that (\ref{telegrafoCondizionatoDatoNsPari}) satisfies the differential equation (\ref{equazioneTelegrafoGenerale}) with time variable $t$ and space variable $y$ and that its integral on $\bigl(x+a_2(t-s),x+a_1(t-s)\bigr)$ is equal to $1-e^{-\lambda (t-s)}$.

\begin{remark}[Symmetric telegraph process]

Let us assume the velocities $a_1=-a_2=c>0$ and the rate $\lambda_1 = \lambda_2=\lambda>0$. Function $g$ in (\ref{funzioneg}) substantially simplifies. We can write $g=x(y-x)f$ where
\begin{align}
x(y-x) f(s,t-s;y-x) = \frac{x(y-x)\,\lambda e^{-\lambda (t-s)}}{2c^2s\sqrt{c^2(t-s)^2-(y-x)^2}}\,I_1\Bigl( \frac{\lambda}{c}\sqrt{c^2(t-s)^2-(y-x)^2}\Bigr)\label{funzionef}
\end{align}
$f$ is positive for $(y-x) \in \bigl(c(t-s),c(t-s)\bigr)$. Then, (\ref{funzionef}) is positive if $sign(y) = sign(x)$ and $|y|>|x|$, so if $y$ is further than $x$ from the origin. This happens because the probability mass of $V(0) = V(s) = sign(x)c$ is greater than the probability of the opposite velocity at time $s$ (see Remark \ref{remarkVDatoTNt}).
\hfill$\diamond$
\end{remark}

\begin{corollary}\label{corollariotcTs}
Let $\{\mathcal{T}(t)\}_{t\ge0}$ be an asymmetric telegraph process. Let $0<s<t,\ x\in(a_2s,a_1s)$. For $y\in \bigl(x+a_2(t-s), x+a_1(t-s)\bigr)$,
\begin{equation}\label{tcTs}
 P \lbrace  \mathcal{T}(t)\in \dif y\ |\ \mathcal{T}(s) =x\rbrace = P\lbrace \mathcal{T}(t-s)\in \dif y-x \rbrace\,+\, g(s,t-s;x,y-x)\dif y\, P\lbrace N(s)\ \text{even} \ |\ \mathcal{T}(s) = x\rbrace ,
\end{equation}
where $g$ is defined in (\ref{funzioneg}). For $v_0\in \{a_1,a_2\}$ and $v_1$ being the other possible velocity,
\begin{equation}\label{tcTssingolarita}
P \lbrace  \mathcal{T}(t) = x+v_0(t-s) \ |\ \mathcal{T}(s) =x\rbrace =  \frac{e^{-\lambda (t-s)}}{2} \Bigl( 1+ \frac{|v_1s-x|}{(a_1-a_2)s}\, P\lbrace N(s)\ \text{even} \ |\ \mathcal{T}(s) = x\rbrace \Bigr).
\end{equation}
\end{corollary}

\begin{proof}
The corollary is an easy consequence of Theorem \ref{corollarioTelegrafoCondizionatoDatoNs}.
\end{proof}

We recall that for $0<s<t,\ x\in(a_2s,a_1s)$, by setting $A(s,x) =\sqrt{(a_1s-x)(x-a_2s)}$, we have
\begin{equation}\label{probabilitaPariCondizionata}
P\lbrace N(s)\ \text{even} \ |\ \mathcal{T}(s) = x\rbrace = \frac{(a_1-a_2)s\,I_1\bigl(\frac{2\lambda}{a_1-a_2}A(s,x) \bigr)}{2A(s,x)\,I_0\bigl(\frac{2\lambda}{a_1-a_2}A(s,x)\bigr) +(a_1-a_2)s\,I_1\bigl(\frac{2\lambda}{a_1-a_2}A(s,x) \bigr)}
\end{equation}
and (\ref{probabilitaPariCondizionata}) converges to $1/2$ under the generalized Kac's conditions. Thus, the asymmetric telegraph process (with $a_1>0>a_2$) at time $t>0$, conditionally on $\mathcal{T}(s) = x$, converges in distribution to the process $B(t-s)+x$ with $\{B(t)\}_{t\ge0}$ being a standard Brownian motion with drift $\mu\in \mathbb{R}$ ($a_1+a_2\longrightarrow-2\mu$). Clearly, (\ref{tcTs}) satisfies the differential equation (\ref{equazioneTelegrafoGenerale}) with time variable $t$ and space variable $y$.
\\

The interested reader can now obtain the joint distribution of the telegraph process at two distinct times as well as the distribution of the telegraph bridge.

\section{Conditional probability of the symmetric telegraph process with its maximum and its minimum}

Let $m(t) := \min_{0\le s \le t} \mathcal{T}(s)$ and $M(t) := \max_{0\le s\le t} \mathcal{T}(s)$ be respectively the minimum and the maximum of the symmetric ($\lambda_1=\lambda_2=\lambda>0,a_1=-a_2=c>0$) telegraph process $\mathcal{T}$, in the time interval $[0,t]$. We denote with $P^{\pm}_{n}(\cdot) = P\{\ \cdot\ |\ V(0)=\pm c,\ N(t)=n\}$ the probability measure conditioned on the starting velocity and the number of switches up to time $t\ge0$.

The conditional and unconditional distributions of the maximum/minimum of the (symmetric) telegraph process are well known, see Foong and Kanno (1994), Cinque and Orsingher (2020, 2021) (for the first passage times see also De Bruyne \textit{et al.} (2021), Ratanov (2021)). The joint distributions with the position of the particle at the ending time are known as well, see Cinque (2020). For the sake of clarity, we observe that, for any integer $n\ge0$, real $x$ and $\beta$, thanks to the symmetry of the motion with constant rate and velocities $\pm c$, we have
\begin{equation}\label{relazioneMaxMin}
P^{\pm}_{n}\{\mathcal{T}(t) \in \dif x ,\ m(t) < -\beta\}= P^{\mp}_{n}\{\mathcal{T}(t) \in -\dif x ,\ M(t) >\beta\}
\end{equation}

Let $\alpha\in [0,ct)$ and $\beta \in [0,ct)$. We focus on the probability that the telegraph process moves both below level $-\alpha$ and above $\beta$ and it ends in $x\in [-\alpha,\beta]$ in the time interval $[0,t]$. Thus, we consider the following distribution, for natural $n \ge 2$ (if $n=0,1$ it is equal to zero),
\begin{align}\label{ptminmax}
P_n^+\{ \mathcal{T}(t) \in \dif x ,\ m(t) <& -\alpha,\ M(t) > \beta \} = P^-_n\{\mathcal{T}(t) \in -\dif x ,\ m(t) < -\beta,\ M(t) > \alpha\} \\
\label{ptmax_min}
&=P_n^+\{ \mathcal{T}(t) \in \dif x ,\ m(t) < -\alpha,\ M(t) > \beta,\ F_{-\alpha} > F_\beta \} \\\label{ptmin_max}
&\ \ \ \ +\ P_n^+\{ \mathcal{T}(t) \in \dif x ,\ m(t) < -\alpha,\ M(t) > \beta,\ F_{-\alpha} < F_\beta \}
\end{align}
with $F_y$ being the first passage time of the telegraph process across level $y\in \mathbb{R}$. We say that the above probabilities are \textit{trivial} if they reduce to known distributions (for instance if $x \in (-ct, -\alpha)\cup(\beta,ct)$).

\begin{proposition}\label{proposizionesupportoptmax_min}
Let $\lbrace \mathcal{T}(t) \rbrace_{t\ge0}$ be a symmetric telegraph process. Distribution (\ref{ptmax_min}) is non-trivial and not null for $(\beta,-\alpha, x)\in \mathcal{S}_M$, where
\begin{equation}\label{supportoptmax_min}
\begin{array}{l}
\mathcal{S}_{M} = \Big\{\;(\beta, -\alpha, x)\,:\, \Bigl(\;\beta \in \bigl[0, \frac{ct}{3}\bigr),\, -\alpha \in \bigl(\frac{3\beta-ct}{2}, 0\bigr],\, x \in \bigl[-\alpha, \beta\bigr] \;\Bigr)\\
\ \ \ \ \ \ \ \ \ \ \ \ \ \ \ \ \ \ \ \ \   or\ \Bigl(\;\beta \in \bigl[0, \frac{ct}{3}\bigr),\, -\alpha \in \bigl(2\beta-ct, \frac{3\beta-ct}{2}\bigr],\, x \in \bigl[-\alpha, ct-2\alpha-2\beta\bigr] \;\Bigr)\\
\ \ \ \ \ \ \ \ \ \ \ \ \ \ \ \ \ \ \ \ \   or\ \Bigl(\;\beta \in \bigl[\frac{ct}{3}, \frac{ct}{2}\bigr),\, -\alpha \in \bigl(2\beta-ct,0\bigr] ,\, x \in \bigl[-\alpha, ct-2\alpha-2\beta\bigr] \;\Bigr)\ \Big\}
\end{array}
\end{equation}
%
and natural $n \ge 2$. Distribution (\ref{ptmin_max}) is non-trivial and not null for $(\beta,-\alpha, x)\in \mathcal{S}_{m}$, where
\begin{equation}\label{supportoptmin_max}
\begin{array}{l}
\mathcal{S}_{m} = \Big\{\;(\beta, -\alpha, x)\,:\, \Bigl(\;\beta \in \big[0, \frac{ct}{2}\bigr),\, -\alpha \in \bigl(\frac{2\beta-ct}{3}, 0\bigr],\, x \in \bigl[-\alpha, \beta\bigr]\;\Bigr)\\
\ \ \ \ \ \ \ \ \ \ \ \ \ \ \ \ \ \ \ \ \ \ \   or\ \Bigl(\;\beta \in \bigl[0, \frac{ct}{2}\bigr),\, -\alpha \in \bigl(\frac{\beta-ct}{2}, \frac{2\beta-ct}{3}\bigr],\, x \in \bigl[2\alpha+2\beta-ct, \beta\bigr] \;\Bigr)\\
\ \ \ \ \ \ \ \ \ \ \ \ \ \ \ \ \ \ \ \ \ \ \   or\ \Bigl(\;\beta \in \bigl[\frac{ct}{2}, ct\bigr),\,-\alpha \in \bigl(\frac{\beta-ct}{2},0\bigr],\, x \in \bigl[2\alpha+2\beta-ct, \beta\bigr] \;\Bigr)\ \Big\}
\end{array}
\end{equation}
and natural $n \ge 3$.
\end{proposition}

\begin{proof}
Since $F_\beta< F_{-\alpha}$, we need $\beta < \frac{ct}{2}$ in order to pass $\beta$ and then move lower than $-\alpha <0$. Thus $ \beta \in [0, \frac{ct}{2})$. Now, we want $\alpha \in [0, ct)$ such that $c(t- \frac{\beta}{c}) > |\beta + \alpha|$, meaning that the motion has time to reach $-\alpha$ after it crossed $\beta$. Thus $-\alpha \in (2\beta- ct, 0]$. At last, we require $x \in [-\alpha, \beta]$ such that $c(t-\frac{\beta}{c}-\frac{\beta+\alpha}{c}) \ge x+\alpha$, therefore $ x \in [-\alpha, \min\{\beta, ct-2\alpha-2\beta \}]$. Finally, $\min\{\beta, ct-2\alpha-2\beta\} = \beta$ if $-\alpha > \frac{3\beta-ct}{2}$ and $2\beta-ct <\frac{3\beta-ct}{2}<0$ if $\beta < \frac{ct}{3}$, then we have (\ref{supportoptmax_min}). 
\\By means of similar arguments we obtain (\ref{supportoptmin_max}).
\end{proof}

We now show a general recurrent formula for probability (\ref{ptmax_min}). We observe the following:
\begin{itemize}
\item[($i$)] if $\mathcal{T}$ moves beyond $\beta$ during the first displacement, it must have time to both reach $-\alpha$ and to be at level $x$ at time $t$, then $T_1\ge \beta/c$ and $c(t-T_1)\ge (cT_1+\alpha)+(x+\alpha)$;
\item[($ii$)] if $T_1<\beta/c$, it is necessary that $-\alpha<\mathcal{T}(T_2) = 2cT_1-cT_2$, because $F_\beta<F_{-\alpha}$, and $c(t-T_2)\ge \beta-\mathcal{T}(T_2)+(\beta+\alpha)+(x+\alpha)$. Therefore, $T_2\le \min\big\{\frac{2cT_1+\alpha}{c},\frac{ct-2\alpha-2\beta-x}{2c}+T_1\big\} = \frac{2cT_1+\alpha}{c}$ if $T_1\le \frac{ct-4\alpha-2\beta-x}{2c}$.
\end{itemize}
With this at hand, by keeping in mind (\ref{telegrafoCondizionatoTempiArrivo}), thanks to a recurrence argument on the displacements of the motion (see for instance the proof of Theorem 3.1 of Cinque and Orsingher (2020)), we can write the following relationship for distribution (\ref{ptmax_min}). Let natural $n\ge2$ and $(\beta,-\alpha, x)\in \mathcal{S}_M$, then
\begin{align}
&P_n^+\lbrace \mathcal{T}(t) \in \dif x ,\ m(t) < -\alpha,\ M(t) > \beta,\ F_{-\alpha} > F_\beta \rbrace  \label{int_max_min}\\
&=\int_{\frac{\beta}{c}}^{\frac{ct-2\alpha-x}{2c}} P^-_{n-1}\{\mathcal{T}(t-t_1) \in \dif x -ct_1 ,\ m(t-t_1) < -\alpha-ct_1\} P\{T_1\in \dif t_1| N(t) = n\} \label{PrimoTermineMaxMin} \\
&\ \ + \begin{cases}
\int_0^{\frac{\beta}{c}} \dif t_1 \int_{t_1}^{\frac{ct-2\alpha-2\beta-x}{2c}+t_1} p^+_{n-2}(t_1,t_2)  \dif t_2,      \hfill{ \text{if} \ \frac{ct-4\alpha-2\beta-x}{2c}\le0\le\frac{\beta}{c}},\\[10pt]
\int_0^{\frac{ct-4\alpha-2\beta-x}{2c}} \dif t_1 \int_{t_1}^{\frac{2ct_1+\alpha}{c}} p^+_{n-2}(t_1,t_2)  \dif t_2 + \int_{\frac{ct-4\alpha-2\beta-x}{2c}}^{\frac{\beta}{c}} \dif t_1\int_{t_1}^{\frac{ct-2\alpha-2\beta-x}{2c}+t_1} p^+_{n-2}(t_1,t_2)  \dif t_2, \\
\hfill{ \text{if} \ 0<\frac{ct-4\alpha-2\beta-x}{2c}\le\frac{\beta}{c}},\\
\int_0^{\frac{\beta}{c}} \dif t_1 \int_{t_1}^{\frac{2ct_1+\alpha}{c}} p^+_{n-2}(t_1,t_2)  \dif t_2, \hfill{ \text{if} \ 0<\frac{\beta}{c}<\frac{ct-4\alpha-2\beta-x}{2c}}, 
\end{cases}\label{SecondoTermineMaxMin}
\end{align}
where 
\begin{align}
p^+_{n-2}(t_1,t_2)&\dif t_1 \dif t_2= P\{T_1\in \dif t_1,T_2\in \dif t_2\ |\ N(t) = n\}\, P_{n-2}^+\{\mathcal{T}(t-t_2)\in \dif x-2ct_1+ct_2,\\
& m(t-t_2)<-\alpha-2ct_1+ct_2, M(t-t_2)> \beta-2ct_1+ct_2,\, F_{-\alpha-2ct_1+ct_2}>F_{\beta-2ct_1+ct_2}\}.\nonumber
\end{align}
We point out that Proposition \ref{proposizionesupportoptmax_min} and the recurrence formula (\ref{int_max_min}) hold true even if we consider a telegraph process with alternating rates $\lambda_1,\lambda_2>0$ (and also for a more general Markovian counting process $N$).
\\

Clearly if $n=2,3$ formula (\ref{int_max_min}) reduces to term (\ref{PrimoTermineMaxMin}) only. By means of (\ref{relazioneMaxMin}) and Theorem 3.1 of Cinque (2020), it is easy to prove that (\ref{PrimoTermineMaxMin}) reads, with natural $n\ge2$,
\begin{equation}\label{minMaxPrimoTermineEsplicito}
\begin{cases}\begin{array}{l l}
\displaystyle\frac{(2k)!}{(k-1)!k!}\frac{(ct+2\alpha+x)^{k-1}(ct-2\alpha-2\beta-x)^{k}}{(2ct)^{2k}}\dif x,  & \text{if}\ n=2k,\\[8pt]
\displaystyle\frac{(2k+1)!}{(k-1)!(k+1)!}\frac{(ct+2\alpha+x)^{k-1}(ct-2\alpha-2\beta-x)^{k+1}}{(2ct)^{2k+1}}\dif x, &\text{if}\ n=2k+1.
\end{array}\end{cases}
\end{equation}

\begin{proposition}
Let $\{\mathcal{T}(t)\}_{t\ge0}$ be a symmetric telegraph process. Let natural $n\ge2$. For $(\beta,-\alpha, x)\in\mathcal{S}_M$ and $\frac{ct-4\alpha-2\beta-x}{2c}\le0\le\frac{\beta}{c}$,
\begin{align}\label{principioRiflessioneNegativoMinMax}
&P_n^+\lbrace \mathcal{T}(t) \in \dif x ,\ m(t) < -\alpha,\ M(t) > \beta,\ F_{-\alpha} > F_\beta \rbrace = P^-_n\{\mathcal{T}(t)\in 2\beta-\dif x,\ M(t)>2\beta+\alpha\}  \\
&=\begin{cases}
\displaystyle P_{2k}^-\{\mathcal{T}(t)\in2\alpha+2\beta +\dif x \}=\frac{(2k)!}{(k-1)!k!}\frac{(ct+2\alpha+2\beta+x)^{k-1}(ct-2\alpha-2\beta-x)^{k}}{(2ct)^{2k}}\dif x,  &\\
 \hfill \text{if}\ n=2k,\\[8pt]
\displaystyle\frac{(2k+1)!}{(k-1)!(k+1)!}\frac{(ct+2\alpha+2\beta+x)^{k-1}(ct-2\alpha-2\beta-x)^{k+1}}{(2ct)^{2k+1}}\dif x, \hfill \text{if}\ n=2k+1.
\end{cases}\nonumber
\end{align}
\end{proposition}

\begin{proof}
Since $\frac{ct-4\alpha-2\beta-x}{2c}\le0\le\frac{\beta}{c}$, the proposition follows by summing up (\ref{minMaxPrimoTermineEsplicito}) and the first case of (\ref{SecondoTermineMaxMin}). The explicit form of the latter is obtained by induction and by suitably applying the results of Corollary 3.4 and Corollary 3.5 of Cinque (2020) (see also Theorem 3.1 of Cinque and Orsingher (2020) for a similar induction method).
\end{proof}

Intuitively, relationship (\ref{principioRiflessioneNegativoMinMax}) holds true because, when $\frac{ct-4\alpha-2\beta-x}{2c}\le0\le\frac{\beta}{c}$, the motion can not reach level $x=-\alpha$ before crossing the threshold $x=\beta$. Formula (\ref{principioRiflessioneNegativoMinMax}) resembles the negative reflection principle for the telegraph process and it can be graphically described in a similar way, see Figures \ref{PRNmin_g1} and \ref{PRNmin_g2} (and see Section 4 of Cinque (2020) for all the details).
\\
\begin{figure}[t]
\begin{minipage}{0.5\textwidth}
\centering
\begin{tikzpicture}[scale = 0.7]
\draw [gray] (6.7,2.9) -- (0,2.9) node[left, black, scale = 1.05]{$\beta$}; \draw[dashed, gray] (6.7, 2.9) -- (7.8,2.9);
\draw [gray] (6.7,-1.55) -- (0,-1.55) node[left, black, scale = 1.05]{$-\alpha$}; \draw[dashed, gray] (6.7, -1.55) -- (7.8,-1.55);
\draw (0,0) -- (0.5, 0.85) -- (0.65, 0.595) -- (2.2,3.23) -- (2.3,3.06) -- (2.45,3.315) -- (3, 2.38)-- (3.15, 2.635)-- (5.7, -1.7)-- (7, 0.51);
\draw[dashed, gray] (2.01,2.9) -- (2.01,0) node[below, black]{$t_1$};
\filldraw[blue] (2.01,2.9) circle (2.5pt) node[above left, black]{\textbf{A}};
\draw[dashed, gray] (2.7,2.9) -- (2.7,0) node[below, black]{$t_2$};
\filldraw[blue] (2.7,2.9)  circle (2.5pt) node[above right, black]{\textbf{B}};
\filldraw[blue] (4.7,0) circle (2.5pt) node[below left, black]{\textbf{C}};
\draw[dashed, gray] (5.61,-1.55) -- (5.61,0) node[above, black]{$t_3$};
\filldraw[blue] (5.61,-1.55) circle (2.5pt) node[below left, black]{\textbf{D}};
\draw[dashed, gray](7,0.51) -- (7,0) node[below, black]{$t$};
\draw[dashed, gray] (7,0.51) -- (0,0.51) node[left, black, scale =1.05]{$x$};
\filldraw[blue]  (7,0.51) circle (2.5pt) node[above, black]{\textbf{E}};
\draw[->, thick] (-1,0) -- (7.8,0) node[below, scale = 1.2]{$\pmb{s}$};
\draw[->, thick] (0,-3) -- (0,6) node[left, scale = 1.1]{ $\pmb{\mathcal{T}(s)}$};
\filldraw[blue] (0,0) circle (3.2pt) node[below left, black]{\textbf{O}};
\end{tikzpicture}
\caption{Sample path $\omega^+$ with $N(t) = 8\ $ and $\frac{ct-4\alpha-2\beta-x}{2c}\le0\le\frac{\beta}{c}$.}\label{PRNmin_g1}
\end{minipage}\hfill
\begin{minipage}{0.495\textwidth}
\centering
\begin{tikzpicture}[scale = 0.65]
\draw [gray] (6.7,2.9) -- (0,2.9) node[left, black, scale = 1.05]{$\beta$}; \draw[dashed, gray] (6.7, 2.9) -- (7.8,2.9);
\draw [lightgray] (6.7,-1.55) -- (0,-1.55) node[left, black, scale = 1.05]{$-\alpha$}; \draw[dashed, lightgray] (6.7, -1.55) -- (7.8,-1.55);
\draw [gray] (6.7,5.8) -- (0,5.8) node[left, black, scale = 1.05]{$2\beta$}; \draw[dashed, gray] (6.7, 5.8) -- (7.8,5.8);
\draw [gray] (6.7,7.35) -- (0,7.35) node[left, black, scale = 1.05]{$2\beta+\alpha$}; \draw[dashed, gray] (6.7, 7.35) -- (7.8,7.35);
\draw[lightgray, thin] (0,0) -- (0.5, 0.85) -- (0.65, 0.595) -- (2.2,3.23) -- (2.3,3.06) -- (2.45,3.315) -- (3, 2.38)-- (3.15, 2.635)-- (5.7, -1.7)-- (7, 0.51);
\draw (0,0)-- (0.19,-0.323) -- (0.29,-0.153) -- (0.44, -0.408) -- (1.18, 0.85) -- (1.33,0.595) -- (3, 3.42)-- (3.15, 3.165)-- (5.7, 7.5)-- (7,5.29);
\filldraw[blue] (0.68,0) circle (2.5pt) node[below right, black]{\textbf{O'}};
\draw[dashed, gray] (2.7,2.9) -- (2.7,0) node[below, black]{$t_2$};
\filldraw[blue] (2.7,2.9)  circle (2.5pt) node[above, black]{\textbf{B}};
\filldraw[blue] (4.7,5.8) circle (2.5pt) node[above, black]{\textbf{C'   }};
\draw[dashed, gray] (5.61,7.35) -- (5.61,0) node[below, black]{$t_3$};
\filldraw[blue] (5.61,7.35) circle (2.5pt) node[above, black]{\textbf{D'}};
\draw[dashed, gray] (7, 5.29) -- (7,0) node[below, black]{$t$};
\draw[dashed, gray] (7, 5.29) -- (0,5.29) node[left, black, scale =1.05]{$2\beta-x$};
\filldraw[blue]   (7, 5.29) circle (2.5pt) node[right, black]{\textbf{E'}};
\draw[->, thick] (-1,0) -- (7.8,0) node[below, scale = 1.2]{$\pmb{s}$};
\draw[->, thick] (0,-1.8) -- (0,8.1) node[left, scale = 1.1]{ $\pmb{\mathcal{T}(s)}$};
\filldraw[blue] (0,0) circle (3.2pt) node[below left, black, scale =1.1]{\textbf{A'}};
\end{tikzpicture}
\caption{The \textit{negatively reflected} sample of $\omega^+$.}\label{PRNmin_g2}
\end{minipage}
\end{figure}

At last, we can write probability (\ref{ptmin_max}) in terms of (\ref{ptmax_min}). Let natural $n\ge 3$ and $(\beta,-\alpha, x)\in \mathcal{S}_{m}$,
\begin{align}\label{int_minmax}
P_n^+ &\{\mathcal{T}(t) \in \dif x ,\ m(t) < -\alpha,\ M(t) > \beta,\ F_{-\alpha} < F_\beta \} =  \int_0^{\min\{\frac{\beta}{c},\frac{ct-2\alpha-2\beta+x}{2c}\}}  P \{ T_1 \in \dif t_1 \ |\ N(t) = n \}\\
&\times P_{n-1}^+\{ \mathcal{T}(t-t_1) \in ct_1-\dif x,\ m(t-t_1) <ct_1-\beta,\ M(t-t_1) > \alpha+ct_1,\ F_{\alpha+ct_1}<F_{ct_1-\beta} \}. \nonumber
\end{align}
By keeping in mind Proposition \ref{proposizionesupportoptmax_min}, we obtain that 
\begin{equation}
\min\Big\{\frac{\beta}{c},\;\frac{ct-2\alpha-2\beta+x}{2c} \Big\} = \frac{\beta}{c}\, \mathds{1}_{\mathcal{R}}(\beta,-\alpha, x)\ +\ \frac{ct-2\alpha-2\beta+x}{2c} \,\mathds{1}_{\mathcal{S}_{m}\setminus\mathcal{R}}(\beta,-\alpha, x) ,
\end{equation}
with $\mathds{1}_A$ being the indicator function of a set $A$ and $\mathcal{R} =\Big\{ (\beta, -\alpha,x)\,:\,\beta \in \bigl[0, \frac{ct}{3}\bigr],\,
-\alpha\in \bigl[\frac{3\beta-ct}{2},0\bigr],\,x\in \bigl[\max\{-\alpha,\ 2\alpha+4\beta-ct\},\, \beta\bigr] \Big\} $.

\appendix
\section{Proof of Theorem \ref{corollarioTelegrafoCondizionatoNsV}}
Probability (\ref{telegrafoCondizionatoNsVSingolarita}) concerns the case where no changes of direction occur in the time interval $(s,t)$, then it immediately follows from Markovianity of $N$.
\\For $0<w<t$, let $\mathcal{T}(t) = \mathcal{T}(w)+\mathcal{T}_w(t)$, where the process $\{\mathcal{T}_w(s)\}_{s\ge w}$ describes the evolution of $\mathcal{T}$ in the time interval $[w,s]$ with respect to the position $\mathcal{T}(w)$, meaning that $\mathcal{T}_w(w)=0\ a.s.$. By bearing in mind (\ref{telegrafoCondizionatoTempiArrivo}), if the motion changes direction at time $w$, then $\mathcal{T}_w(t) \stackrel{d}{=} \mathcal{T}(t-w)$. Thus,
\begin{align}
P\{ &\mathcal{T}(t) \in \dif y\ |\ \mathcal{T}(s) = x, N(s) = k, N(t) = n, V(0) = v_0  \} \nonumber\\
& = P\{\mathcal{T}(s) + v_k(T_{k+1}-s) + \mathcal{T}_{T_{k+1}}(t)\in \dif y \ |\ \mathcal{T}(s) = x, N(s) = k, N(t) = n, V(0) = v_0\}\nonumber\\
&=\int_s^{t} P\{\mathcal{T}_w(t)\in \dif y-x-v_k(w-s) \ |\  N(w) = k+1, N(s)=k, N(t) = n, V(w) = v_{k+1} \}\nonumber\\
&\hspace{1cm}\times P\{T_{k+1}\in\dif w\ |\ N(s) = k, N(t) = n, V(s) = v_k\}  \nonumber\\
&=\int_s^{t} P\{\mathcal{T}(t-w)\in \dif y-x-v_k(w-s) \ |\  N(t-w) = n-k-1, V(0) = v_{k+1} \} \nonumber\\
&\hspace{1cm}\times P\{T_1\in\dif w-s\ |\ N(t-s) = n-k, V(0) = v_k\} \nonumber\\
&=\int_0^{t-s} P\{\mathcal{T}(t-s-t_1)\in \dif y-x-v_kt_1\ |\  N(t-s-t_1) = n-k-1, V(0) = v_{k+1} \}\nonumber\\
&\hspace{1cm}\times P\{T_1\in\dif t_1\ |\ N(t-s) = n-k, V(0) = v_k\}.\label{condizionataDimostrazione1}
\end{align}
The third equality follows from Markovianity of $N$ and the considerations on $\mathcal{T}_w$. The probability of the telegraph process in the integral of (\ref{condizionataDimostrazione1}) is $0$ if $t_1\not \in\bigl(0,\frac{|v_{k+1}(t-s)-(y-x)|}{a_1-a_2}\bigr)$, which follows since the density in the integral is positive if $\,a_2(t-s-t_1)<y-x-v_kt_1 <a_1(t-s-t_1)\implies \begin{cases}
t_1(v_k-a_2)< -a_2(t-s)+(y-x), \\
t_1(a_1-v_k)< a_1(t-s)-(y-x) .
\end{cases}
$
By replacing $v_k = a_1$ and $v_k =a_2$, simple algebra leads to the condition $0<t_1 <\frac{|v_{k+1}(t-s)-(y-x)|}{a_1-a_2} < t-s$. Hence, we have that
\begin{align}
 P\{& \mathcal{T}(t) \in \dif y\ |\ \mathcal{T}(s) = x, N(s) = k, N(t) = n, V(0) = v_0  \}  \label{telegrafoCondizionatoNsNtV}\\
&=\int_0^{\frac{|v_{k+1}(t-s)-(y-x)|}{a_1-a_2}}  P\{T_1\in\dif t_1\ |\ N(t-s) = n-k, V(0) = v_k\}\nonumber\\
&\hspace{1.2cm}\times P\{\mathcal{T}(t-s-t_1)\in \dif y-x-v_kt_1\ |\  N(t-s-t_1) = n-k-1, V(0) = v_{k+1} \} \nonumber\\
& = P \{ \mathcal{T}(t-s) \in \dif y-x\ |\ N(t-s) = n-k,\ V(0) =  v_k\},\nonumber
\end{align}
where the last equality follows by means of a recurrence argument on the distribution of the telegraph process based on (\ref{telegrafoCondizionatoTempiArrivo}). Finally, with (\ref{telegrafoCondizionatoNsNtV}) at hand, (\ref{telegrafoCondizionatoNsV}) follows from the law of total probability.

\end{document}